\documentclass[12pt]{amsart}
\usepackage{amssymb}
\usepackage{verbatim}
\usepackage{hyperref}

\begin{document}

% define theorem environments
\newtheorem{theorem}{Theorem}    %[section]
\newtheorem{proposition}[theorem]{Proposition}
\newtheorem{conjecture}[theorem]{Conjecture}
\def\theconjecture{\unskip}
\newtheorem{corollary}[theorem]{Corollary}
\newtheorem{lemma}[theorem]{Lemma}
\newtheorem{sublemma}[theorem]{Sublemma}
\newtheorem{observation}[theorem]{Observation}
\theoremstyle{definition}
\newtheorem{definition}{Definition}
\newtheorem{notation}[definition]{Notation}
\newtheorem{remark}[definition]{Remark}
\newtheorem{question}[definition]{Question}
\newtheorem{questions}[definition]{Questions}
\newtheorem{example}[definition]{Example}
\newtheorem{problem}[definition]{Problem}
\newtheorem{exercise}[definition]{Exercise}

\numberwithin{theorem}{section}
\numberwithin{definition}{section}
\numberwithin{equation}{section}

\def\earrow{{\mathbf e}}
\def\rarrow{{\mathbf r}}
\def\uarrow{{\mathbf u}}
\def\tpar{T_{\rm par}}
\def\apar{A_{\rm par}}

\def\reals{{\mathbb R}}
\def\torus{{\mathbb T}}
\def\heis{{\mathbb H}}
\def\integers{{\mathbb Z}}
\def\naturals{{\mathbb N}}
\def\complex{{\mathbb C}\/}
\def\distance{\operatorname{distance}\,}
\def\support{\operatorname{support}\,}
\def\dist{\operatorname{dist}\,}
\def\Span{\operatorname{span}\,}
\def\degree{\operatorname{degree}\,}
\def\kernel{\operatorname{kernel}\,}
\def\dim{\operatorname{dim}\,}
\def\codim{\operatorname{codim}}
\def\trace{\operatorname{trace\,}}
\def\Span{\operatorname{span}\,}
\def\ZZ{ {\mathbb Z} }
\def\p{\partial}
\def\rp{{ ^{-1} }}
\def\Re{\operatorname{Re\,} }
\def\Im{\operatorname{Im\,} }
\def\ov{\overline}
\def\eps{\varepsilon}
\def\lt{L^2}
\def\diver{\operatorname{div}}
\def\curl{\operatorname{curl}}
\def\etta{\eta}
\newcommand{\norm}[1]{ \|  #1 \|}
\def\Span{\operatorname{span}}
\def\expect{\mathbb E}
\def\paraboloid{{\mathbb P}^2}
\def\Sbest{{\mathbf S}}
\def\Pbest{{\mathbf P}}

\newcommand{\Norm}[1]{ \left\|  #1 \right\| }
\newcommand{\set}[1]{ \left\{ #1 \right\} }
\def\one{\mathbf 1}
\newcommand{\modulo}[2]{[#1]_{#2}}

\def\scriptf{{\mathcal F}}
\def\scriptg{{\mathcal G}}
\def\scriptm{{\mathcal M}}
\def\scriptb{{\mathcal B}}
\def\scriptc{{\mathcal C}}
\def\scriptt{{\mathcal T}}
\def\scripti{{\mathcal I}}
\def\scripte{{\mathcal E}}
\def\scriptv{{\mathcal V}}
\def\scriptS{{\mathcal S}}
\def\scripta{{\mathcal A}}
\def\scriptr{{\mathcal R}}
\def\scripto{{\mathcal O}}
\def\scripth{{\mathcal H}}
\def\scriptd{{\mathcal D}}
\def\scriptl{{\mathcal L}}
\def\scriptn{{\mathcal N}}
\def\frakv{{\mathfrak V}}

\author{Michael Christ}
\address{
        Michael Christ\\
        Department of Mathematics\\
        University of California \\
        Berkeley, CA 94720-3840, USA}
\email{mchrist@math.berkeley.edu}
\thanks{The first author was supported in part by NSF grant DMS-0901569.}
%The second author was supported by the National Science Foundation under agreement ??DMS-0635607??.
%Any opinions, findings, and conclusions
%or recommendations expressed in this paper are those of the authors
%and do not necessarily reflect the views of the National Science Foundation.}
%s DMS-0401260 and

\author{Shuanglin Shao}
\address{Shuanglin Shao\\
%School of Mathematics, Institute for Advanced Study, Princeton, NJ 08540
%\\
IMA, University of Minnesota, Minneapolis, MN 55455
}
\email{slshao@ima.umn.edu}

\date{% October 6, 2009.}
June 17, 2010.}

\title[Extremizers of a Fourier restriction inequality]
{On the extremizers of an  \\ adjoint Fourier restriction inequality}

\begin{abstract}
The adjoint Fourier restriction inequality for the sphere $S^2$
states that if $f\in\lt(S^2,\sigma)$ then $\widehat{f\sigma}\in
L^4(\reals^3)$. We prove that all critical points $f$ of the
functional $\norm{\widehat{f\sigma}}_{L^4}/\norm{f}_{\lt}$
are smooth; that any complex-valued extremizer
for the inequality is a nonnegative extremizer multiplied
by the character $e^{ix\cdot\xi}$ for some $\xi$; and
that complex-valued
extremizing sequences for the inequality are precompact
modulo multiplication by characters.
\end{abstract}

\maketitle
%{\Small \tableofcontents}

\section{Results}

Let $S^2$ denote the unit sphere in $\reals^3$, equipped
with surface measure $\sigma$. The adjoint Fourier restriction
inequality states that
there exists $C<\infty$ such that
\begin{equation} \label{inequalityR}
\norm{\widehat{f\sigma}}_{L^4(\reals^3)}
\le C\norm{f}_{\lt(S^2,\sigma)}
\end{equation}
for all $f\in\lt(S^2)$.
With the Fourier transform defined to be
$\widehat{g}(\xi) = \int e^{-ix\cdot\xi}g(x)\,dx$,
denote by
\begin{equation}
\scriptr
=\sup_{0\ne f\in\lt(S^2)} \norm{\widehat{f\sigma}}_{L^4(\reals^3)}
\ \big/\ \norm{f}_{\lt(S^2,\sigma)}
\end{equation}
the optimal constant in the inequality \eqref{inequalityR}.

In an earlier paper \cite{christshao1} we have proved that
there exists $f\in\lt$ which extremizes this inequality,
and that any sequence of nonnegative functions $\{f_\nu\}\subset\lt(S^2)$
satisfying $\norm{f_\nu}_2\to 1$
and $\norm{\widehat{f_\nu\sigma}}_4\to\scriptr$
is precompact in $\lt(S^2)$.
In the present paper we prove that all extremizers are infinitely differentiable,
and show that precompactness does continue to hold
for complex-valued extremizing sequences, modulo the action of
a natural noncompact symmetry group of the inequality.

\eqref{inequalityR}
is equivalent, by Plancherel's theorem, to
\begin{equation} \label{secondSversion}
\norm{f\sigma*f\sigma}_{\lt(\reals^3)} \le \Sbest^2\norm{f}_{\lt(S^2)}^2,
\end{equation}
where
$\scriptr= (2\pi)^{3/4}\Sbest$
and $*$ denotes convolution of measures.

\begin{definition}
An extremizing sequence for the inequality \eqref{inequalityR}
is a sequence $\{f_\nu\}$ of functions in $\lt(S^2)$
satisfying $\norm{f_\nu}_2\le 1$
such that $\norm{\widehat{f_\nu\sigma}}_{L^4(\reals^3)}\to\scriptr$
as $\nu\to\infty$.

An extremizer for the inequality \eqref{inequalityR}
is a function $f\ne 0$ which satisfies $\norm{\widehat{f\sigma}}_4=\scriptr\norm{f}_2$.
\end{definition}

Define the functional
\begin{equation}
\Lambda(f)=\norm{\widehat{f\sigma}}_4^4/\norm{f}_2^4.
\end{equation}
A real-valued function
$0\ne f\in\lt(S^2)$ is a critical point of $\Lambda$ if and only
if $f$ satisfies the generalized Euler-Lagrange equation
\begin{equation} \label{eulerlagrange}
\Big(f\sigma*f\sigma*f\sigma\Big)\Big|_{S^2} = \lambda \norm{f}_2^2\, f
\ \ \ \text{ almost everywhere on } S^2
\end{equation}
for some scalar $\lambda\in\reals^+$.
See for instance \cite{christquilodran},
where a more general result of this type is proved.
$f$ is an extremum for $\Lambda$ if and only
if this holds with $\lambda = \Sbest^4$.

\begin{theorem} \label{thm:smoothness}
For any $\lambda\in\complex$,
any solution $f\in\lt(S^2)$
of \eqref{eulerlagrange} is $C^\infty$.
\end{theorem}
Thus any real-valued critical point,
and in particular any nonnegative extremizer, of $\Lambda$ is $C^\infty$.
It is possible to show by a straightforward iteration argument
that there exists a Gevrey class which contains all critical points,
but we have not been able to show that these are real analytic.

\begin{theorem}\label{thm:complexextremizers}
Every
complex-valued extremizer for the
inequality \eqref{inequalityR}
is of the form
\begin{equation}ce^{ix\cdot\xi}F(x)\end{equation}
where
$\xi\in \reals^3$, $c\in\complex$,
and $F$ is a nonnegative extremizer.
\end{theorem}
Thus all complex-valued extremizers are $C^\infty$,
as well.

\begin{theorem} \label{thm:complexsequences}
If $\{f_\nu\}$ is any complex-valued
extremizing sequence, then there exists a sequence
$\{\xi_\nu\}\subset\reals^3$
such that $\{e^{-ix\cdot\xi_\nu}f_\nu(x)\}$ is precompact.
\end{theorem}

\section{Smoothness of critical points}
\label{section:smoothness}

For $\alpha\in(0,1)$ denote by $\Lambda_\alpha$ the space of all
H\"older continuous functions of order $\alpha$ on $S^2$,
with norm
\begin{equation}
\norm{f}_{\Lambda_\alpha}
= \norm{f}_{C^0}
+ \sup_{x\ne x'} |x-x'|^{-\alpha}|f(x)-f(x')|.
\end{equation}
$H^s=H^s(S^2)$ will denote the usual Sobolev space of
functions having $s\ge 0$ derivatives in $\lt$.
$H^0$ will be synonymous with $\lt$.

\begin{lemma}
For any $s\ge 0$
there exists a constant $A_s<\infty$ such that
for any functions $h_j\in H^s(S^2)$,
\begin{equation} \label{usetogetcontraction}
\norm{(h_1\sigma*h_2\sigma*h_3\sigma)\big|_{S^2} }_{H^s}
\le A_s\norm{h_1}_{H^s}\norm{h_2}_{H^s}\norm{h_3}_{H^s}.
\end{equation}

Moreover, for $s$ in any compact subinterval of $[0,\infty)$,
\eqref{usetogetcontraction} holds with a constant $A$ independent of $s$.
A corresponding bound holds in the spaces $\Lambda_\alpha$
for all $0\le\alpha<1$, with a constant independent of $\alpha$.
\end{lemma}
\noindent
The proofs of these routine inequalities are left to the reader.

The following is one of two main steps in the proof of
Theorem~\ref{thm:smoothness}.
\begin{lemma}  \label{lemma:trilinearHs}
Let $a:S^2\to\complex$ be any complex-valued function
which is H\"older continuous of some positive order.
Then for
any solution $f\in H^0(S^2)$ of the equation
\begin{equation} \label{eulerlagrangegeneralized}
f(x) = a(x)(f\sigma*f\sigma*f\sigma)(x) \text{ for almost every } x\in S^2,
\end{equation}
there exists $s>0$ such that $f\in H^s(S^2)$.

Let
$\{f_\nu\}$ be a family
of solutions of \eqref{eulerlagrangegeneralized}
with coefficient functions $a=a_\nu$.
If $\norm{f_\nu}_{\lt}=1$ for all $\nu$,
if the functions $a_\nu$ have uniformly bounded
$\Lambda_\alpha$ norms for some $\alpha>0$,
and if $\{f_\nu\}$ is precompact in $\lt(S^2)$,
then there exist $B<\infty$ and $s>0$
such that $\norm{f_\nu}_{H^s}\le B$ uniformly for all $\nu$.
\end{lemma}

Note that precompactness in $H^0$ is a hypothesis for the
second part of the lemma, not a conclusion. In an earlier paper we have
proved that nonnegative extremizing sequences for the functional
$\norm{\widehat{f\sigma}}_{\lt}^4/\norm{f}_{\lt}^4$
are precompact, but we have not established any corresponding result
for arbitrary critical
points satisfying the Euler-Lagrange equation with uniformly bounded constant
Lagrange multipliers $a$.

The functional
$\norm{\widehat{f\sigma}}_{\lt}^4/\norm{f}_{\lt}^4$
is essentially scale-invariant
at small scales.
Therefore it is not true that for any $f\in H^0(S^2)$,
$(f\sigma*f\sigma*f\sigma)\big|_{S^2}\in H^s$ for some $s>0$.
Thus a straightforward bootstrapping argument cannot establish
the smoothness of all solutions. But
any particular solution is not scale-invariant, and therefore
breaks the (approximate) scaling symmetry.
Because any solution breaks the symmetry in its own way,
the proof yields an exponent $s$ which is not universal, but
depends on the critical point itself.

\begin{proof}
Let $f\in \lt(S^2)$ satisfy the equation for
some function $a\in \Lambda_\alpha(S^2)$.
For any $\eps\in (0,1]$,
$f$ may be decomposed
as $f=\varphi_\eps+g_\eps$ where
$\varphi_\eps\in C^\infty$,
$\norm{g_\eps}_{\lt}<\eps$,
and $\norm{\varphi_\eps}_{\lt}\le C\norm{f}_{\lt}$,
where $C<\infty$ is independent of $\eps$.

Reformulate the equation by substituting $f=\varphi_\eps+g_\eps$
for all four occurrences of $f$. Express the result in the form
\begin{equation}
g_\eps = \scriptl(\varphi_\eps,g_\eps) +\scriptn(\varphi_\eps,g_\eps)
\end{equation}
where
\begin{align}
\scriptl(\varphi_\eps,g_\eps) &= -\varphi_\eps
+ a\cdot(\varphi_\eps\sigma*\varphi_\eps\sigma*\varphi_\eps\sigma)
+ 3a\cdot(\varphi_\eps\sigma*\varphi_\eps\sigma*g_\eps\sigma)
\\
\scriptn(\varphi_\eps,g_\eps)
&= 3a\cdot(\varphi_\eps\sigma*g_\eps\sigma*g_\eps\sigma)
+ a\cdot(g_\eps\sigma*g_\eps\sigma*g_\eps\sigma).
\end{align}
$\scriptl(\varphi_\eps,g_\eps)$ and
$\scriptn(\varphi_\eps,g_\eps)$ are regarded as elements of $\lt(S^2)$,
rather than of $\lt(\reals^3)$.

For the ``linear'' term $\scriptl(\varphi_\eps,g_\eps)$ there are two useful bounds.
Firstly,
\begin{equation}
\norm{\scriptl(\varphi_\eps,g_\eps)}_{\Lambda_\alpha}\le \norm{\varphi_\eps}_{\Lambda_\alpha}
+ C\norm{\varphi_\eps}_{\Lambda_\alpha}^3
+ C \norm{\varphi_\eps}_{\Lambda_\alpha}^2
\norm{g_\eps}_{\lt}^2
\end{equation}
where $C$ depends on $\norm{a}_{\Lambda_\alpha}$.
$\Lambda_\alpha$ embeds continuously in $H^\alpha$,
so
$\scriptl(\varphi_\eps,g_\eps)\in H^\alpha$  and
\begin{equation}
\norm{\scriptl(\varphi_\eps,g_\eps)}_{H^\alpha} \le C(\eps)\text{for all $\eps>0$,}
\end{equation}
where $C(\eps)<\infty$ but we have no useful upper bound.
Secondly,  since
\begin{equation}
\norm{\scriptn(\varphi_\eps,g_\eps)}_{\lt(S^2)}
\le C\norm{\varphi_\eps}_{\lt}\norm{g_\eps}_{\lt}^2
+ C\norm{g_\eps}_{\lt}^3,
\end{equation}
the representation $\scriptl(\varphi_\eps,g_\eps)= g_\eps - \scriptn(\varphi_\eps,g_\eps)$ gives
\begin{equation}
\norm{\scriptl(\varphi_\eps,g_\eps)}_{H^0} \le \norm{g_\eps}_{H^0}
+ C\norm{g_\eps}_{H^0}^2 + C\norm{g_\eps}_{H^0}^3
\le C\eps.
\end{equation}

A consequence is that if $\eps$ is first chosen
to be sufficiently small, and if $s(\eps)>0$ is subsequently chosen to be sufficiently
small as a function of $\norm{\varphi_\eps}_{H^\alpha}$, which in turn depends on $\eps$,
then
\begin{equation}
\norm{\scriptl(\varphi_\eps,g_\eps)}_{H^{s(\eps)}}<\eps^{7/8}.
\end{equation}
This is obtained by interpolating between the favorable $H^0$ bound,
and the potentially unfavorable $H^\alpha$ bound.
Since $\norm{\varphi_\eps}_{H^0}$ is bounded above uniformly in $\eps$,
by choosing first $\eps$ small, then $s(\eps)$ sufficiently small
we may ensure in the same way that
\begin{equation} \label{varphinotsobad}
\norm{\varphi_\eps}_{H^{s(\eps)}}\le \eps^{-1/4}.
\end{equation}

For each $\eps\in(0,1]$ define the operator
\begin{equation}
L_\eps(h) =
\scriptl(\varphi_\eps,g_\eps) + \scriptn(\varphi_\eps,h)
\end{equation}
for $h\in\lt(S^2)$.
$L_\eps$ maps $H^s(S^2)$ continuously to itself for all $s\in[0,\alpha]$,
by Lemma~\ref{lemma:trilinearHs}.

Denote by $B= B(\scriptl(\varphi_\eps,g_\eps),\eps^{3/4})$
the ball  of radius $\eps^{3/4}$ in $H^{s(\eps)}(S^2)$ centered at $\scriptl(\varphi_\eps,g_\eps)$.
By \eqref{usetogetcontraction}
and the bounds
$\norm{\scriptl(\varphi_\eps,g_\eps)}_{H^{s(\eps)}}<\eps^{7/8}$
and $\norm{\varphi_\eps}_{H^s(\eps)}<\eps^{-1/4}$,
if $\eps$ is sufficiently small
then $L_\eps$
maps $B$ to itself,
and is a strict contraction on $B$.
Indeed, if $\scriptn(\varphi_\eps,h)-\scriptn(\varphi_\eps,\tilde h)$ is expanded
in the natural way, then a typical term of the worst type which results is
$a\cdot\big(\varphi_\eps\sigma*h\sigma*(h-\tilde h)\sigma\big)$.
For $s=s(\eps)$,
its $H^{s}$ norm is majorized by
\[
C\norm{\varphi_\eps}_{H^s}\norm{h}_{H^s}\norm{h-\tilde h}_{H^s}
\le C\eps^{-1/4}\eps^{3/4}
\norm{h-\tilde h}_{H^s}
\ll
\norm{h-\tilde h}_{H^s}.
\]

Therefore for any sufficiently small $\eps>0$
there exists a solution $h_\eps\in H^{s(\eps)}$
of $h_\eps= L_\eps(h_\eps)$,
satisfying $\norm{h_\eps}_{H^{s(\eps)}}\le \eps^{3/4}$.
Moreover, there exists only one solution satisfying this norm bound.
The same reasoning applies, and therefore the
same uniqueness holds, with $H^{s(\eps)}$ replaced by $H^0$.
Since the $H^{s(\eps)}$ norm majorizes the $\lt$ norm,
if $\eps$ is sufficiently
small then
$h_\eps$ is also the unique $H^{0}$ solution with small $H^{0}$ norm.
We know that $g_\eps$ is a solution with small $H^0$ norm,
so $g_\eps=h_\eps$,
and thus $g_\eps\in H^{s(\eps)}$.
Specializing to any single such value of $\eps$
gives the first conclusion of the lemma.

This argument suffices to establish the uniform version
stated above, as well. Since $\{f_\nu\}$
is precompact, $f_\nu$ may be decomposed as $f_\nu=\varphi_\nu+g_\nu$
where $\varphi_\nu,g_\nu$ depend also on $\eps$ and satisfy
$\norm{g_\nu}_{\lt}<\eps$
and $\norm{\varphi}_{C^1}\le C_\eps$,
where $C_\eps<\infty$ is independent of $\nu$.
The proof then proceeds as above, with all quantities uniform in $\nu$.
\end{proof}

The second main step in the proof of regularity is a routine
bootstrapping procedure.
We have found it to be convenient to carry this procedure out in the following
function spaces $\scripth^s$.
For $0\le s\notin\integers$,
write $C^s=C^{k,\alpha}$
for $s\in (k,k+1)$ for each nonnegative integer $k$.
Then to $f\in\lt(S^2)$ associate
$F(\Theta,x)$ defined by
$F(\Theta,x) = f(\Theta(x)) = (\Theta f)(x)$
for $(\Theta,x)\in O(3)\times S^2$.
For $0\le s\notin\integers$
define $\scripth^s$ to be the set
of all $f\in\lt(S^2)$ whose lift $F$ belongs to $C^s_\Theta L^2_x(O(3)\times S^2)$.
The norm for this space is
\begin{equation}
\norm{f}_{\scripth^s}
=
\norm{f}_{\lt(S^2)} +
\sup_{\Theta\ne I}
|\Theta-I|^{-s}\norm{\Theta f-f}_{\lt(S^2)},
\end{equation}
where $|\Theta-I|$ denotes the distance from $\Theta$
to the identity matrix, with respect to any fixed metric on $O(3)$.

Of course, the mappings $f\mapsto\Theta(f)$
map $\scripth^s$ boundedly to $\scripth^s$,
uniformly for all $\Theta\in O(3)$, for all $s$.
\begin{lemma}
For any $\eps>0$ there exists $\delta>0$
such that $f\sigma*g\sigma*h\sigma\in \scripth^\delta$
whenever $f,g\in\scripth^\eps$ and $h\in H^0$,
with
\begin{equation}
\norm{f\sigma*g\sigma*h\sigma}_{\scripth^\delta}
\le C_{\eps}\norm{f}_{\scripth^\eps}
\norm{g}_{\scripth^\eps}
\norm{h}_{H^0}.
\end{equation}
\end{lemma}

\begin{proof}
Write for $z\in\reals^3$
\begin{equation}
(h\sigma*f\sigma*g\sigma)(z)
= \int_{S^2} h(y) (f\sigma*g\sigma)(z-y)\,d\sigma(y).
\end{equation}
Therefore for $\Theta\in O(3)$,
\begin{multline} \label{Theta-I}
(\Theta-I)(h\sigma*f\sigma*g\sigma)(z)
\\
= \int_{S^2} h(y)
\Big((f\sigma*g\sigma)(\Theta(z)-y)-(f\sigma*g\sigma)(z-y)\Big)\,d\sigma(y).
\end{multline}
If $f,g$ are Lipschitz functions on $S^2$
then $f\sigma*g\sigma(x)$ is the product of a function
in $\Lambda_{1/2}(\reals^3)$
of $x$ with $|x|^{-1}\chi_{|x|\le 2}$.
When  \eqref{Theta-I}
is calculated for $z\in S^2$, only $y$ satisfying $|y|\le 2$ come into
play.
Thus this integral takes the form
\begin{equation}
\int_{S^2}K(z,y)|z-y|^{-1}h(y)\,d\sigma(y)
\end{equation}
where $K\in\Lambda_{1/2}(S^2\times S^2)$.
It is routine to verify that such a
linear transformation maps $\lt(S^2)$
to $\scripth^\delta$
for some $\delta>0$.

If $f\in \scripth^\eps$ then
for any $\eta>0$,
$f$ may be decomposed as $f=f^\sharp+f^\flat$
where $\norm{f^\flat}_{H^0}\le\eta$
and $\norm{f^\sharp}_{\text{Lip} 1}
\le C\eta^{-C}$,
where $C=C(\eps)<\infty$.
From this and the above result for
Lipschitz $f,g$ it follows that
for all $f,g\in \scripth^\eps$ and $h\in\lt$,
$(\Theta-I)(h\sigma*f\sigma*g\sigma)\in\scripth^\delta$
for a smaller exponent $\delta=\delta(\eps)>0$.
This concludes the proof for $s\in (0,1)$.

For $s=k+\alpha$ with $\alpha\in (0,1)$,
we first differentiate $F(\Theta,x)$ $k$
times with respect to $\Theta$, then invoke the
case $\alpha\in (0,1)$ for each of the resulting terms.
\end{proof}

\begin{lemma}
Let $a\in C^\infty(S^2)$.
For any $\eps>0$ there exists $\delta>0$
such that for any
$s\in [\eps,\infty)\setminus\integers$ and
any function
$f\in \scripth^s(S^2)$,
\begin{equation}
a\cdot(f\sigma*f\sigma*f\sigma)\Big|_{S^2}
\in\scripth^{t}(S^2)
\text{ for all } t\in[0,s+\delta]\setminus\integers.
\end{equation}
\end{lemma}

\begin{proof}
Consider $s=\alpha\in(0,1)$.
The factor $a(x)$ is harmless.
We write $f\sigma*f\sigma*f\sigma$
as shorthand for
$(f\sigma*f\sigma*f\sigma)\Big|_{S^2}$,
where convenient.
For $\Theta\in O(3)$,
\begin{equation}
\big(\Theta-I\big)(f\sigma*f\sigma*f\sigma)
=
(\Theta-I)(f)\sigma*\Theta f\sigma*\Theta f\sigma
+
f\sigma*(\Theta-I)f\sigma*\Theta f\sigma
+
f\sigma*f\sigma*(\Theta-I)f\sigma.
\end{equation}
Now for
$\delta>0$ sufficiently small,
\begin{equation}
\norm{(\Theta-I)f\sigma*f\sigma*f\sigma}_{\scripth^\delta}
\le C
\norm{(\Theta-I)f}_{H^0}\norm{f}_{\scripth^s}^2
\le C|\Theta-I|^s \norm{f}_{\scripth^s}^3.
\end{equation}
The same applies to the other two terms, so
\begin{equation}
\norm{(\Theta-I)\big(f\sigma*f\sigma*f\sigma\big)}_{\scripth^\delta}
\le C|\Theta-I|^s \norm{f}_{\scripth^s}^3.
\end{equation}
Therefore
\begin{equation}
\norm{(\Theta-I)^2\big(f\sigma*f\sigma*f\sigma\big)}_{H^0}
\le C|\Theta-I|^{s+\delta} \norm{f}_{\scripth^s}^3.
\end{equation}
By the classical characterization
of H\"older spaces of orders in $(0,1)\cup(1,2)$ in terms
of second differences, this implies that
$(f\sigma*f\sigma*f\sigma)\in \scripth^{s+\delta}$.
\end{proof}

We finish by establishing another property of nonnegative extremizers.
\begin{lemma} \label{lemma:strictlypositive}
Let $a\in C^0(S^2)$ satisfy $a(x)>0$ for all $x\in S^2$.
Let $f\in C^0(S^2)$ be any continuous, nonnegative, even
solution of
$f = a\cdot(f\sigma*f\sigma*f\sigma)\big|_{S^2}$
which does not vanish identically.
Then $f(x)>0$ for  every $x\in S^2$.
\end{lemma}

\begin{proof}
There exists $x_0\in S^2$ for which $f(x_0)>0$.
Since $f(-x_0)=f(x_0)$, $f$ is continuous, and $f\ge 0$ everywhere,
this forces there to exist
a neighborhood of $0$ in which $f\sigma*f\sigma$ is uniformly bounded below
by some strictly positive number.
Therefore $a\cdot(f\sigma*f\sigma*f\sigma)\ge f\sigma *K$
for some nonnegative function $K\in C^0(\reals^3)$ which
satisfies $K(0)>0$.
The inequality $f\ge f\sigma*K$ forces $f>0$ everywhere.
\end{proof}

\begin{corollary} \label{cor:nonvanishing}
For any
nonnegative extremizer $0\ne f\in\lt(S^2)$
of the functional $\norm{\widehat{f\sigma}}_4^4/\norm{f}_{\lt}^4$
there exists $\delta>0$
such that $f(x)\ge\delta$ for almost every $x\in S^2$.
\end{corollary}

Indeed, it was proved in \cite{christshao1}
that any such extremizer is necessarily an even function.
It was shown above that $f\in C^\infty$.
Thus the hypotheses of Lemma~\ref{lemma:strictlypositive}
are satisfied.

\section{Complex-valued extremizers}

\begin{proof}[Proof of Theorem~\ref{thm:complexextremizers}]
Denote by $B(0,2)$ the ball centered at the origin
of radius $2$ in $\reals^3$.
Let $0\ne f\in\lt(S^2)$
be a complex extremizer and write
\begin{equation}f=e^{i\varphi}F \end{equation}
where $\varphi$ is real-valued and measurable, and $F=|f|$
is a nonnegative extremizer.
Trivially $|(f\sigma*f\sigma)(z)|\le (F\sigma*F\sigma)(z)$
for almost every $z\in\reals^3$.
By Corollary~\ref{cor:nonvanishing}, $(F\sigma*F\sigma)(z)>0$
for almost every $z\in B(0,2)$, and of course $\equiv 0$
whenever $|z|>2$.
Therefore $f$ is an extremizer if and only if
\begin{equation}
|(f\sigma*f\sigma)(z)| = (F\sigma*F\sigma)(z)
\ \ \text{for almost every $z\in B(0,2)$.}
\end{equation}

For any $z\in\reals^3$ satisfying $0<|z|<2$,
there exists a singular positive measure $\mu_z$ on
$S^2\times S^2$, supported on $\{(x,y): x+y=z\}$,
satisfying
\begin{equation}
(h_1\sigma*h_2\sigma)(z) = \int h_1(x)h_2(y)\,d\mu_z(x,y)
\end{equation}
for arbitrary $h_1,h_2$.
Moreover,
for almost every $z$,
the relation
$|f\sigma*f\sigma(z)|=(F\sigma*F\sigma)(z)>0$
forces
$e^{i\varphi(x)}e^{i\varphi(y)}$ to depend only on $z$
for $\mu_z$--almost every pair $(x,y)$.
Therefore for $\sigma\times\sigma$--almost every $(x,y)\in S^2$,
\begin{equation}
\text{
$e^{i[\varphi(x)+\varphi(y)]}$ depends only on $x+y$.}
\end{equation}
Therefore
there exists a measurable real-valued function $\psi$, defined
for almost every $z\in B(0,2)$, satisfying
\begin{equation}
(f\sigma*f\sigma)(z)
= e^{i\psi(z)}(F\sigma*F\sigma)(z),
\end{equation}
that is,
\begin{equation} \label{eq:varphileadstotopsi}
e^{i(\varphi(x)+\varphi(y))} = e^{i\psi(x+y)}
\end{equation}
for $\sigma\times\sigma$ almost every $(x,y)\in S^2\times S^2$.

We aim to prove that $\psi$ has the form
$\psi(z)=ce^{iz\cdot\xi}$ for almost every $z\in B(0,2)$,
for some $c\in\complex$ satisfying $|c|=1$
and some $\xi\in\reals^3$.
From \eqref{eq:varphileadstotopsi}
it follows directly that $\varphi$ has the same form, almost everywhere on $S^2$.

\begin{definition}
\begin{equation}
\Lambda=\{\vec{z}=(z_1,z_2,z_3,z_4)\in(\reals^3)^4: z_1+z_2=z_3+z_4\}.
\end{equation}
$\Lambda$ is a smooth manifold of dimension $9$.
$\lambda$ denotes the natural ``surface'' measure on
$\Lambda$ induced from its inclusion
into $(\reals^{3})^4$.
\end{definition}

\begin{lemma} \label{lemma:separatevariables}
Let
$\vec{\bar z}=(\bar z_1,\bar z_2,\bar z_3,\bar z_4)\in\Lambda$.
Suppose that there exists a neighborhood $U\subset\Lambda$
of $\vec{\bar z}$ such that
\begin{equation}
e^{i[\psi(z_1)+\psi(z_2)]}
=e^{i[\psi(z_3)+\psi(z_4)]}
\ \ \text{for $\lambda$-almost every $\vec{z}\in U$.}
\end{equation}
Then there exist $\xi\in\reals^3$
and a constant $c\in\complex$ satisfying $|c|=1$
and a neighborhood $V\subset\reals^3$ of $\bar z_1$
such that for Lebesgue almost every $w\in V$,
\begin{equation}e^{i\psi(w)}=ce^{iw\cdot\xi}.\end{equation}
\end{lemma}
\noindent This lemma will be proved below.

If for every $\bar w\in B(0,2)$ there exist $c,\xi$ such that
$e^{i\psi(w)}\equiv ce^{iw\cdot\xi}$ for almost every $w$
in some neighborhood of $\bar w$,
then $c,\xi$ must clearly be independent of $\bar w$,
so
$e^{i\psi(w)}\equiv ce^{iw\cdot\xi}$ for almost every $w\in B(0,2)$.
Thus we aim to prove that $\psi$ is additive
in the sense that for every $\bar z_1\in B(0,2)\subset\reals^3$,
there exist $\vec{\bar z}$ and a neighborhood $U$
satisfying the hypothesis of Lemma~\ref{lemma:separatevariables}.

\begin{definition}
$G\subset S^2\times S^2$ is
\begin{equation}
G=\{(x,y)\in S^2\times S^2: x\ne\pm y \text{ and }
e^{i[\varphi(x)+\varphi(y)]}=e^{i\psi(x+y)}\}.
\end{equation}
$\Omega\subset (S^2)^4\times (S^2)^4$  is defined by
\begin{equation}
\Omega=\{
(\vec{x},\vec{y})=(x_1,\cdots,y_4)\in (S^2)^8:
x_1+x_2=y_3+y_4
\text{ and }
x_3+x_4=y_1+y_2\}.
\end{equation}
$\pi:\Omega\to\Lambda$ is  the mapping
\begin{equation}
\pi(\vec{x},\vec{y}) = (x_1+y_1,x_2+y_2,x_3+y_3,x_4+y_4).
\end{equation}
\end{definition}

We know that
\begin{equation}(\sigma\times\sigma)((S^2\times S^2)\setminus G)=0.\end{equation}
$\Omega$ is a $16-6=10$-dimensional real algebraic variety, with singularities.
The two equations defining $\Omega$ ensure that $\pi(\Omega)\subset\Lambda$.
$\Omega$ is equipped
with a natural ``surface'' measure $\rho$
which is supported on the set of all smooth points of $\Omega$,
and is induced from
$\sigma\times\cdots\times\sigma$,
via the inclusion of $\Omega$ into $(S^2)^8$.

\begin{lemma} \label{lemma:psiadditive}
Let $\vec{z}$ in $\Lambda$,
and suppose that there exists
$(\vec{x},\vec{y})\in\Omega$
such that $\pi(\vec{x},\vec{y})=\vec{z}$,
$(x_j,y_j)\in G$ for all  $j\in\{1,2,3,4\}$,
and
$(x_1,x_2),\,
(x_3,x_4),\,
(y_1,y_2),\,
(y_3,y_4) $ all belong to $G$ as well.
Then
\begin{equation}
e^{i[\psi(z_1)+\psi(z_2)]}
=e^{i[\psi(z_3)+\psi(z_4)]}.
\end{equation}
\end{lemma}

\begin{proof}
$e^{i\psi(z_j)} = e^{i[\phi(x_j)+\phi(y_j)]}$
for each $j$ by definition of $\psi$ since $(x_j,y_j)\in G$.
Therefore
\begin{align*}
e^{i[\psi(z_1)+\psi(z_2)-\psi(z_3)-\psi(z_4)]}
&=
e^{ i[\phi(x_1)+\phi(y_1)]}
e^{ i[\phi(x_2)+\phi(y_2)]}
e^{-i[\phi(x_3)+\phi(y_3)]}
e^{-i[\phi(x_4)+\phi(y_4)]}
\\
&=
e^{ i[\phi(x_1)+\phi(x_2)]}
e^{-i[\phi(y_3)+\phi(y_4)]}
\cdot
e^{ i[\phi(y_1)+\phi(y_2)]}
e^{-i[\phi(x_3)+\phi(x_4)]}.
\end{align*}
Since $(x_1,x_2)\in G$,
$(y_3,y_4)\in G$,
and $x_1+x_2=y_3+y_4$,
$
e^{ i[\phi(x_1)+\phi(x_2)]}
=
e^{i[\phi(y_3)+\phi(y_4)]}
$.
Similarly
$e^{ i[\phi(y_1)+\phi(y_2)]}
=
e^{i[\phi(x_3)+\phi(x_4)]}
$.
Thus the product equals $1$.
\end{proof}

\begin{lemma}  \label{lemma:submerse}
Suppose that $(\vec{\bar x},\vec{\bar y})\in \Omega$
satisfies
\begin{equation} \label{notparallel}
\begin{aligned}
&\bar x_j\ne\pm \bar y_j \text{ for all $j\in\{1,2,3,4\}$,}
\\
&\bar x_3\ne \pm \bar x_4,\ \
\bar y_3\ne\pm \bar y_4.
\end{aligned}
\end{equation}
Then
$(\vec{\bar x},\vec{\bar y})$
is a smooth point of $\Omega$.

If in addition
\begin{equation} \label{submersehypothesis}
\Span(x_1,y_1)^\perp +\Span(x_2,y_2)^\perp +\Span(x_3,y_3)^\perp +\Span(x_4,y_4)^\perp
=\reals^3,
\end{equation}
then $\pi:\Omega\to\Lambda$ is a submersion at
$(\vec{\bar x},\vec{\bar y})$.
\end{lemma}
\noindent This lemma will be proved below.

Let
$(\vec{\bar x},\vec{\bar y})$ satisfy the hypotheses of Lemma~\ref{lemma:submerse}.
Since $\pi$ is a submersion at
$(\vec{\bar x},\vec{\bar y})$,
there exist neighborhoods $U\subset\Omega$
of $(\vec{\bar x},\vec{\bar y})$
and $V\subset\Lambda$ of $\vec{\bar z}=\pi(\vec{\bar x},\vec{\bar y})$
such that
$\pi(U)\supset V$, and moreover,
\begin{equation} \label{mutuallyAC}
\text{
The measures
$(\pi_*(\rho|_{U}))\big|_V$
and $\lambda|_V$
are mutually absolutely continuous.}
\end{equation}
Here $\mu|_E$ denotes the restriction of a measure $\mu$ to a measurable set $E$,
and $\pi_*(\rho|_U)(E) = \rho(U\cap \pi^{-1}(E))$.

Define $\Omega^\natural$ to be the set of all $(\vec{x},\vec{y})\in\Omega$
which  satisfy \eqref{submersehypothesis}
and $x_i\ne\pm x_j\ne\pm y_k$ for all $i,j,k\in\{1,2,3,4\}$
with $i\ne j$,
and for which each pair $(x_j,y_j)$ lies in $G$,
and each of the pairs
$(x_1,x_2),(x_3,x_4),(y_1,y_2),(y_3,y_4)$ also lies in $G$.
In a neighborhood of any point of $\Omega$,
any two of the eight two-dimensional
variables $x_i,y_j$ give $4$ independent coordinates. It follows that
$\rho(\Omega\setminus\Omega^\natural)=0$.
By \eqref{mutuallyAC},
since the image under $\pi$ of a $\rho$-null set is a $\pi_*(\rho)$-null set,
the measures
$(\pi_*(\rho|_{U\cap \Omega^\natural}))\big|_V$
and $\lambda|_V$
are again mutually absolutely continuous.

By Lemma~\ref{lemma:psiadditive},
this implies that for any $(\vec{x},\vec{y})\in\Omega^\natural$,
$e^{i[\psi(\zeta_1)+\psi(\zeta_2)-\psi(\zeta_3)-\psi(\zeta_4)]}=1$
for $\lambda$--almost every $\vec{\zeta}=(\zeta_1,\zeta_2,\zeta_3,\zeta_4)\in\Lambda$
in some neighborhood of $\vec{z} = \pi(\vec{x},\vec{y})\in\Lambda$.

In combination with the next lemma,
this completes the proof of Theorem~\ref{thm:complexextremizers}.
\end{proof}

\begin{lemma} \label{lemma:existenceofpreimages}
For any $(w_1,w_2)\in B(0,2)\times B(0,2)$
with $0<|w_1|,|w_2|<2$
there exists $(\vec{x},\vec{y})\in\Omega^\natural$
satisfying $x_j+y_j=w_j$ for both $j=1$ and $j=2$.
\end{lemma}

\section{Proofs of auxiliary lemmas}

\begin{proof}[Proof of Lemma~\ref{lemma:existenceofpreimages}]
The set of all solutions $(x_1,y_1)\in (S^2)^2$
of $x_1+y_1=w_1$ is a certain circle,
and the condition $0<|w_1|<2$ ensures that
$x_1\ne\pm y_1$ for all such points.
There is a corresponding circle of points $(x_2,y_2)$ satisfying
$x_2+y_2=w_2$, and once $(x_1,y_1)$ has been specified,
any generic pair of this type
satisfies $x_2,y_2\ne\pm x_1,y_1$.
Once $(x_1,x_2,y_1,y_2)$ are specified, the pairs $(y_3,y_4)$
which satisfy $y_3+y_4=x_1+x_2$ form another circle,
and again, any generic point of this circle
satisfies the constraints $y_3,y_4\notin\{\pm x_1,\pm x_2,\pm y_1,\pm y_2\}$.
Finally $(x_3,x_4)$ may also be chosen in the same way to satisfy
$x_3,x_4\notin\{\pm x_1,\pm x_2,\pm y_j\}$.
\end{proof}

\begin{proof}[Proof of Lemma~\ref{lemma:separatevariables}]
It suffices to prove the following:
Let $\psi$ be a real-valued measurable function in two nonempty open
sets $U,V\subset\reals^d$. Suppose that
$e^{i[\psi(z)+\psi(w)]}$ equals a function
of $z+w$ alone for Lebesgue-almost every $(z,w)\in U\times V$.
Then there exist $\xi\in\reals^d$
and $c\in\complex$ such that
$e^{i\psi(z)}\equiv ce^{iz\cdot\xi}$ for almost every $z\in U$.

Given any two distributions in $\scriptd'(U\times V)$ which
depend respectively only on $z,w$ in the natural sense,
their product is well-defined as a distribution.
Moreover
\begin{equation}
(\nabla_z-\nabla_w)\big(e^{i\psi(z)+i\psi(w)}\big)
=
e^{i\psi(w)}\cdot
\nabla e^{i\psi(z)}
-e^{i\psi(z)} \cdot \nabla e^{i\psi(w)}
\end{equation}
in the sense of distributions.
The hypothesis that $e^{i\psi(z)}e^{i\psi(w)}$
depends only on $z+w$ means that the left-hand
side vanishes identically, as a distribution.
By pairing the right-hand side with test functions
$f(z)g(w)$
and fixing any $g\in\scriptd(V)$ such that $\langle g,e^{i\psi}\rangle\ne 0$,
we conclude that there exist $c_1,c_2\in\complex$ with $c_1\ne 0$ such that
\begin{equation}
c_1 \nabla e^{i\psi(z)}=c_2 e^{i\psi(z)}
\end{equation}
in $\scriptd'(U)$.
Therefore $e^{i\psi}$ takes the required form.
\end{proof}

\begin{proof}[Proof of Lemma~\ref{lemma:submerse}]
Formally, the tangent space to $\Omega$ at a point
$(\vec{x},\vec{y})$ is the vector space of
all $(\vec{u},\vec{v})\in(\reals^3)^8$
which satisfy $u_j\perp x_j$ and $v_j\perp y_j$ for
$j\in\{1,2,3,4\}$,
$u_1+u_2=v_3+v_4$,
and
$v_1+v_2=u_3+u_4$.
This can be written as a system of $14$ scalar equations
for $24$ variables.
By the implicit function theorem,
$\Omega$ is a smooth $10$-dimensional manifold
in a neighborhood of any point for which this associated
vector space has the maximum possible dimension, $10$.

Writing $v_4=u_1+u_2-v_3$ and $u_4 = v_1+v_2-u_3$,
the relations $v_4\perp y_4$ and $u_4\perp x_4$
become inhomogeneous linear equations for $u_3,v_3$
in terms of $u_1,u_2,v_1,v_2$. It suffices to show
that for each $(u_1,u_2,v_1,v_2)$ satisfying
$u_j\perp x_j$ and $v_j\perp y_j$,
the set of all solutions $(u_3,v_3)$
of the four equations $u_3\perp x_3$, $u_4\perp x_4$, $v_3\perp y_3$,
and $v_4\perp y_4$
is an affine two-dimensional space.
Equivalently, we wish the mapping
$(u_3,v_3)\mapsto (u_3\cdot x_3,u_3\cdot x_4,v_3\cdot y_3,v_3\cdot y_4)$
to have a nullspace of dimension exactly two.
The conditions $x_3\ne \pm x_4$ and
$y_3\ne\pm y_4$ ensure this since $x_i,y_j\ne 0$.

Next, let $(\vec{x},\vec{y})\in\Omega$ satisfy \eqref{submersehypothesis}.
We wish to show that $\pi:\Omega\to\Lambda$ is a submersion at $(\vec{x},\vec{y})$.
The range of $D\pi$ on the associated tangent spaces
is the set of all $(u_1+v_1,\cdots,u_4+v_4)\in(\reals^3)^4$
where $(\vec{u},\vec{v})$ varies over the  space described above.
The tangent space of $\Lambda$ is the vector space
of all $w\in(\reals^3)^4$ which satisfy
$w_1+w_2=w_3+w_4$.
We will show that
for any $w\in (\reals^3)^4$,
there exists $(\vec{u},\vec{v})$
satisfying $u_j\perp x_j$ and $v_j\perp y_j$ for all $j$,
$u_1+u_2=v_3+v_4$,
and
$u_j+v_j=w_j$ for all $j$.
If $w$ satisfies
the tangency condition
$w_1+w_2=w_3+w_4$
is satisfied, then
\[
v_1+v_2-u_3-u_4
= (w_1+w_2-w_3-w_4)-(u_1+u_2-v_3-v_4)
=0-0
=0.
\]

Because $x_j\ne\pm y_j$,
each of the four equations $u_j+v_j=w_j$,
together with the constraints $u_j\perp x_j$
and $v_j\perp y_j$,
allows $u_j$ to vary freely over
a certain translate of the one-dimensional space
$\Span(x_j,y_j)^\perp$, and
specifies $v_j$ uniquely as a function of $u_j$.
Each can alternatively be regarded as allowing
$v_j$ to vary freely over a translate of
$\Span(x_j,y_j)^\perp$, and
specifying $u_j$ uniquely as a function of $v_j$.
Therefore we can solve for $v_1,v_2,u_3,u_4$
in terms of $(\vec{w}, u_1,u_2,v_3,v_4)$,
as $u_1,u_2,v_3,v_4$ each vary freely over
the appropriate one-dimensional affine subspace.

The only equation remaining to be satisfied is
$u_1+u_2-v_3-v_4=0$.
As $u_1,u_2,v_3,v_4$ vary freely over the allowed
affine spaces,
the function
$u_1+u_2-v_3-v_4$
takes on a constant value,
plus any element of
$\Span(x_1,y_1)^\perp +\Span(x_2,y_2)^\perp +\Span(x_3,y_3)^\perp +\Span(x_4,y_4)^\perp$.
Since the sum of these four spaces is assumed to equal $\reals^3$,
this function
$u_1+u_2-v_3-v_4$
has range $\reals^3$.
In particular, $0$ belongs to its range;
there does exist a solution of
$u_1+u_2-v_3-v_4=0$ satisfying the above constraints.

Thus there exists a solution of the given
system of equations for $(\vec{u},\vec{v})$.
Therefore $\pi$ is indeed a submersion at $(\vec{x},\vec{y})$.
\end{proof}

The following more quantitative result will be needed below in the analysis
of complex-valued extremizing sequences.

\begin{proposition} \label{prop:quantitative}
For any $\eps>0$ there exists $\delta>0$ with the following property.
Let $\scriptg\subset S^2\times S^2$ satisfy
$(\sigma\times\sigma)(S^{2+2}\setminus\scriptg)<\delta$.
Let $\varphi:S^2\to\reals$ and $\psi:B(0,2)\to\reals$
be measurable functions
satisfying
$|e^{i[\varphi(x)+\varphi(x')]}-e^{i\psi(x+x')}|<\delta$
for all $(x,x')\in\scriptg$.
Then there exist a set
$\scripte\subset B(0,2)\times B(0,2)$ satisfying $|\scripte|<\eps$ and
a measurable function $h:B(0,4)\to\complex$
such that
for all $(z,z')\in \big(B(0,2)\times B(0,2)\big)\setminus\scripte$,
\begin{equation}
\big|
e^{i[\psi(z)+\psi(z')]}-h(z+z')
\big|
<\eps.
\end{equation}
\end{proposition}

\begin{proof}
Let $\eta>0$. If $\delta$ is sufficiently small
then there exists $\scripte_1\subset B(0,2)$
such that $|\scripte_1|<\eta$,
and $B(0,2)\setminus\scripte_1$ is contained in a union of
$N(\eta)<\infty$ disks $V_\alpha$
such that for each $\alpha$,
$V_\alpha\times V_\alpha$
is a neighborhood in $B(0,2)^2$ of a point $(z,z)$
for which there exists $(\vec{x},\vec{y})\in\Omega$
such that
$\pi(\vec{\bar x},\vec{\bar y})=(\bar z_1,\bar z_2,\bar z_3,\bar z_4)$
satisfies
$\bar z_1=\bar z_2=z$.
More precisely, $V_\alpha$ is sufficiently small
that $\pi$ is a submersion of a neighborhood $U_\alpha$
of
$(\vec{\bar x},\vec{\bar y})\in\Omega$
onto
a neighborhood of $(\bar z_1,\bar z_2,\bar z_3,\bar z_4)$ in $\Lambda$.
The mutual absolute continuity of $(\pi_*(\rho|_{U_\alpha}))\big|_{V_\alpha}$
and $\lambda|_{V_\alpha}$, together with the smallness of $(S^2\times S^2)\setminus\scriptg$,
imply that
for most $\vec{z}=(z_1,z_2,z_3,z_4)$ in $\pi(U_\alpha)$,
there exists $(\vec{x},\vec{y})\in U_\alpha$ satisfying $\pi(\vec{x},\vec{y})=\vec{z}$,
$(x_j,y_j)\in\scriptg$ for $j\in\{1,2,3,4\}$,
and
$(x_1,x_2),(x_3,x_4),(y_1,y_2),(y_3,y_4)$
all belong to $\scriptg$ as well.
Here ``most'' means that the set $E_\alpha$ of all
$\vec{z}\in\pi(U_\alpha)$
which lack such a representation
satisfies $\lambda(E_\alpha)<\eta/N(\eta)$,
provided that $\delta$ is chosen to be a sufficiently small function of $\eta$.

Define $S_\alpha$ to be the set of all $\vec{z}\in\pi(U_\alpha)$
which admit such a representation.
It follows from the proof of Lemma~\ref{lemma:psiadditive}
that
\begin{equation}
\big|e^{i[\psi(z_1)+\psi(z_2)-\psi(z_3)-\psi(z_4)]}-1\big|
=O(\delta)
\end{equation}
for all $\vec{z}\in S_\alpha$.

Define $T_\alpha$ to be the set of all
$(z_1,z_2,z'_1,z'_2)\in V_\alpha^4$
for which there exist $z_3,z_4$ such that both
$(z_1,z_2,z_3,z_4)$
and
$(z'_1,z'_2,z_3,z_4)$
belong to $S_\alpha$.
Such points satisfy $z_1+z_2=z'_1+z'_2$,
that is, $T_\alpha\subset\Lambda$.
Again
\begin{equation}
\big|e^{i[\psi(z_1)+\psi(z_2)-\psi(z'_1)-\psi(z'_2)]}-1\big|
=O(\delta)
\end{equation}
for all
$(z_1,z_2,z'_1,z'_2)\in T_\alpha$.
Moreover,
$\lambda((\Lambda\cap V_\alpha^4)\setminus T_\alpha)\to 0$ as $\delta\to 0$.

There exist
a measurable function $h_\alpha:V_\alpha\times V_\alpha\to\complex$
and
a function $\theta(\delta)$
which tends to zero as $\delta\to 0$,
such that
\begin{equation} \label{whathdoes}
\big|
e^{i[\psi(z_1)+\psi(z_2)]}-h(z_1+z_2)
\big| \le\theta(\delta)
\end{equation}
for all $(z_1,z_2)\in V_\alpha^2$, except for a subset of $V_\alpha^2$
whose measure is $\le\theta(\delta)$.
The function $\theta$ may be taken to depend only on $\delta$,
not in any other way on $\psi$.
Indeed, for $w\in V_\alpha+V_\alpha$, $h(w)$
may be defined to be the average value of $e^{i[\psi(z_1)+\psi(z_2)]}$,
where this average is taken over
$\{(z_1,z_2)\in V_\alpha^2: z_1+z_2=w\}$
with respect to the natural Lebesgue measure on that set.
As $\lambda\big((\Lambda\cap V_\alpha^4)\setminus T_\alpha\big)\to 0$,
the Lebesgue measure of the set of all $(z_1,z_2)\in V_\alpha^2$ which
fail to satisfy \eqref{whathdoes} tends to zero.
\end{proof}

\section{On approximate characters}

We seek to analyze functions
$\phi:S^2\to\reals$
for which
$e^{i[\phi(x)+\phi(x')]}$
is well approximated by a function
of $x+x'\in\reals^3$ alone,
for almost every pair $(x,x')\in S^2$.
In this section we study a more basic question of the
same type, in which the domain of the phase
function $\phi$ is an open set in $\reals^3$,
rather than a null set such as $S^2$.
By an approximate character in $\reals^3$, we mean
a real-valued function $\psi$
such that $e^{i[\psi(x)+\psi(y)]}$
is nearly equal to a function of $x+y$,
for nearly all pairs $(x,y)$ in an open set
in $\reals^3\times\reals^3$.
In this section we characterize approximate characters.
In the next section, the result will be applied to the analysis of functions
$\phi$ which nearly satisfy the functional equation only on
the null set $S^2\times S^2$.

\begin{proposition} \label{prop:character}
Let $D\subset\reals^d$ be any bounded disk.
For any $\eps>0$ there exists $\delta>0$ with the following property.
Let $\psi:D\to\reals$  and $h:D+D\to\complex$ be measurable functions
which satisfy
\begin{equation} \label{nearcharacter}
|\{(x,y)\in D\times D: \big|e^{i[\psi(x)+\psi(y)]}-h(x+y) \big|> \delta\}|<\delta.
\end{equation}
Then
there exist $\xi\in\reals^d$ and $c\in\complex$ satisfying $|c|=1$
such that
\begin{equation}
\norm{e^{i\psi(x)}-ce^{ix\cdot\xi}}_{\lt(D)}<\eps.
\end{equation}
\end{proposition}

\begin{proof}
By a change of variables $x\mapsto a+rx$ we may assume that
$D$ is the unit disk centered at $0$.
We may assume without loss of generality that $|h(x+y)|=1$
for all $x+y\in D+D=2D$.
Define $h(x)=0$ for all $|x|>2$.

For $t\in \reals^d$ let
$\lambda_t$ denote Lebesgue measure on $\{(x,y)\in\reals^{d+d}: x+y=t\}$.
Define
\begin{align*}
f(x) &= e^{i\psi(x)}
\\
g(x,y) &= e^{i[\psi(x)+\psi(y)]}-h(x+y)
\\
G(t) &= \int_{x+y=t}g(x,y)\,d\lambda_t(x,y).
\end{align*}
Since $|h|\equiv 1$ on $2D$ and $|f|\equiv 1$,
$|g|\le 2$ and thus, by \eqref{nearcharacter},
\begin{equation}
\norm{G}_{\lt(\reals^d)}\to 0 \text{ as $\delta\to 0.$}
\end{equation}
Likewise define
\[
H(t) = \int_{x+y=t}h(x+y)\,d\lambda_t(x,y)
= h(t)\int_{x+y=t}\,d\lambda_t(x,y).
\]
$\norm{H}_{\lt(\reals^d)}$ is bounded above
by a constant independent of $\psi$.
Moreover, $\norm{H}_{\lt(\reals^d)}$ is bounded
below by a positive constant, independent of $\psi$.
$G,H$ vanish identically on the complement of $2D$.

For any $\eta\in\reals^d$,
\begin{align}
\widehat{f}(\eta)^2
&= \iint_{D^2}
e^{-i(x+y)\eta}e^{i[\psi(x)+\psi(y)]}\,dx\,dy
\\
&= \widehat{g}(\eta,\eta)
+ \iint_{D^2}
e^{-i(x+y)\eta}h(x+y)\,dx\,dy
\\
&=
\widehat{G}(\eta)+
\widehat{H}(\eta)
\end{align}
since
\begin{equation}
\widehat{g}(\eta,\eta) = \int e^{-it\cdot\eta}G(t)\,dt = \widehat{G}(\eta).
\end{equation}

Therefore, since $\norm{H}_2$ is uniformly positive
and $\norm{G}_2\to 0$ as $\delta\to 0$,
whenever $\delta$ is sufficiently small then
$\norm{(\widehat{f})^2}_{\lt}^2=\int_{\reals^d}|\widehat{f}(\eta)|^4\,d\eta$
is bounded below by a constant which depends only on the dimension $d$.
Since
\begin{gather*}
\int |\widehat{f}(\eta)|^2\,d\eta = (2\pi)^d \norm{f}_{\lt}^2=(2\pi)^d |D|
\\
\int_{\reals^d}|\widehat{f}(\eta)|^4\,d\eta\le \norm{\widehat{f}}_{L^\infty}^2
\norm{\widehat{f}}_{\lt}^2,
\end{gather*}
we conclude that there exist $c_0,c_1>0$ such that
if $\delta\le c_1$ then there exists $\zeta\in\reals^d$
such that
\begin{equation}
|\widehat{f}(\zeta)|\ge c_0.
\end{equation}
By replacing $\psi(x)$ by $\psi(x)-x\cdot\zeta$ we may and will assume that $\zeta=0$,
and thus that
$|\widehat{f}(0)|\ge c_0$.

Next, for any $\xi\in\reals^d$,
\begin{align}
\widehat{f}(\xi)\widehat{f}(0)
&= \iint_{D\times D} f(x)f(y)e^{-i\xi\cdot x}\,dx\,dy
\\
&= \iint_{D\times D}
e^{-i(x+y)\cdot\xi/2}
e^{-i(x-y)\cdot\xi/2}
h(x+y)\,dx\,dy
+
\widehat{g}(\xi,0)
\\
&=
\int h(t)e^{-it\cdot\xi/2}K(t,\xi)\,dt
+\widehat{g}(\xi,0)
\end{align}
where
\begin{equation}
K(t,\xi) =
\int_{x+y=t}
e^{-i(x-y)\cdot\xi/2}
\,d\lambda_t(x,y),
\end{equation}
with the restriction $(x,y)\in D^2$ in this integral.
The set of all $(x,y)\in D^2$ satisfying $x+y=t$
is naturally identified with a disk in $\reals^{d}$
of radius $\le 1$.
It is routine to verify that
\begin{equation}
|K(t,\xi)|\le C(1+|\xi|)^{-(d+1)/2}
\end{equation}
uniformly for all $t\in 2D$ and $\xi\in\reals^d$,
where $C<\infty$ depends only on the radius of $D$.
Therefore
\begin{equation}
\big|\int h(t)e^{-it\cdot\xi/2}K(t,\xi)\,dt\big|
\le
C(1+|\xi|)^{-(d+1)/2}.
\end{equation}
Thus there is an upper bound
\begin{equation}
\big|
\widehat{f}(\xi)\widehat{f}(0)
\big|
\le
C(1+|\xi|)^{-(d+1)/2}
+
C|\widehat{g}(\xi,0)|.
\end{equation}
Since $|\widehat{f}(0)|\ge c_0$,
this implies that
\begin{equation}
|\widehat{f}(\xi)|
\le
C(1+|\xi|)^{-(d+1)/2}
+
C|\widehat{g}(\xi,0)|,
\end{equation}
uniformly for all $\xi\in\reals^d$.

Now since $g$ is supported in the bounded set $D^2$,
\[
\int_{\reals^d} |\widehat{g}(\xi,0)|^2\,d\xi\le C\norm{g}_{\lt}^2
\le C\delta.
\]
Thus for any $R\ge 1$,
\begin{equation} \label{eq:precompactfouriertrick}
\int_{|\xi|\ge R}
\big|\widehat{f}(\xi)\big|^2\,d\xi
\le CR^{-1}+ C\delta.
\end{equation}

In order to prove Proposition~\ref{prop:character},
it suffices to prove the following: For
any sequence of functions $\psi_\nu$
satisfying the hypothesis with a sequence of constants
$\delta_\nu$ which tend to zero as $\nu\to\infty$,
there exist $c_\nu,\xi_\nu$
such that
$\norm{e^{i\psi_\nu(x)}-c_\nu e^{i\xi_\nu\cdot x}}_{\lt(D)}\to 0$
for some sequence of indices $\nu$ tending to $\infty$.

Let $\{\psi_\nu\}$ be such a sequence.
As shown above, by \eqref{eq:precompactfouriertrick}
there exists a sequence $\{\eta_\nu\}\subset\reals^d$
such that the set of functions $f_\nu (x) = e^{i[\psi_\nu(x)-\eta_\nu\cdot x]}$
is precompact in $\lt(D)$.
Passing to a convergent subsequence, we obtain
$f\in\lt(D)$ such that $\norm{f_\nu-f}_{\lt(D)}\to 0$.
Since $|f_\nu|\equiv 1$, $|f|\equiv 1$ as well,
so
$f(x)=e^{i\psi(x)}$ for some measurable real-valued function $\psi$.

For any $j\in \{1,2,\cdots,d\}$,
let $L_j$ denote
the partial differential operator $\partial_{x_j}-\partial_{y_j}$,
which acts on functions and distributions defined on open subsets of $\reals^{d+d}$.
For each index $\nu$,
write
\begin{equation}
e^{i[\psi_\nu(x)+\psi_\nu(y)]}
= h_\nu(x+y)+g_\nu(x,y).
\end{equation}
Thus
\begin{equation}
f_\nu(x)f_\nu(y) = e^{-i\eta_\nu\cdot(x+y)}h_\nu(x+y) + e^{-i\eta_\nu\cdot(x+y)}g_\nu(x,y)
= \tilde h_\nu(x+y)+\tilde g_\nu(x,y).
\end{equation}
Then $L_j(\tilde h_\nu)\equiv 0$,
and $L_j(\tilde g_\nu)\to 0$ in $H^{-1}(\reals^{d+d})$ as $\nu\to\infty$
since $\tilde g_\nu\to 0$ in $H^0$.
Therefore $L_j(f_\nu(x)f_\nu(y))\to 0$ in $H^{-1}(\reals^{d+d})$.
Therefore $L_j(f(x)f(y))\equiv 0$, in the sense of distributions.

Since this holds for each index $j$,
$f(x)f(y)$ must depend only on $x+y$, for almost every $(x,y)\in D\times D$.
This forces
$f(x)=e^{i\psi(x)}=ce^{ix\cdot\xi}$ for some $\xi\in\reals^d$
and some unimodular constant $c\in\complex$.
Thus
\begin{equation}
e^{i[\psi_\nu(x)-\eta_\nu\cdot x]}\to ce^{ix\cdot\xi} \text{ in $\lt(D)$.}
\end{equation}
Equivalently,
\begin{equation}
\norm{e^{i\psi_\nu(x)}-ce^{i(\xi+\eta_\nu)\cdot x}}_{\lt(D)}\to 0,
\end{equation}
as was to be proved.
\end{proof}
\section{Complex extremizing sequences}

Let $\{f_\nu\}$ be a sequence of complex-valued functions in $\lt(S^2)$
which satisfy $\norm{f_\nu}_2\to 1$
and $\norm{f_\nu\sigma*f_\nu\sigma}_{\lt(\reals^3)}\to \Sbest^2$
as $\nu\to\infty$.
Write $f_\nu = e^{i\varphi_\nu}F_\nu$ where $F_\nu=|f_\nu|$.

Define $\delta_\nu\ge 0$ by
$\norm{f_\nu\sigma*f_\nu\sigma}_{\lt(\reals^3)}=(1-\delta_\nu)^2 \Sbest^2$.
Then
$\delta_\nu\to 0$ as $\nu\to\infty$, and
$\norm{F_\nu\sigma*F_\nu\sigma}_{\lt(\reals^3)}\ge (1-\delta_\nu)^2 \Sbest^2$.

\begin{lemma} \label{lemma:fromphitopsi}
There exist measurable functions $\psi_\nu:B(0,2)\to\reals$
and positive numbers $\eta_\nu$
such that for each $\nu$,
\begin{equation}
\big|
e^{i[\varphi_\nu(x)+\varphi_\nu(x')]}
-
e^{i\psi_\nu(x+x')}
\big|
<\eta_\nu
\end{equation}
for all $(x,x')\in S^{2+2}$
except for a set whose $\sigma\times\sigma$ measure
is $<\eta_\nu$.
\end{lemma}
\noindent A proof will be indicated below.

The proof of Theorem~\ref{thm:complexsequences} is concluded
by combining Lemma~\ref{lemma:fromphitopsi}
with ingredients developed above.
By Proposition~\ref{prop:quantitative},
there exist measurable functions $h_\nu:B(0,4)\to\complex$,
positive numbers $\eps_\nu$,
and measurable sets $\scripte_\nu\subset B(0,2)^2$
such that
$\eps_\nu\to 0$ and $|\scripte_\nu|\to 0$ as $\nu\to\infty$,
and
for all $(z,z')\in \big(B(0,2)\times B(0,2)\big)\setminus\scripte_\nu$,
$\big|
e^{i[\psi(z)+\psi(z')]}-h(z+z')
\big|
<\eps_\nu$.
By Proposition~\ref{prop:character},
there exist $\xi_\nu\in\reals^3$ and $c_\nu\in\complex$ satisfying $|c_\nu|=1$
such that
\begin{equation}
\norm{e^{i\psi_\nu(x)}-c_\nu e^{ix\cdot\xi_\nu}}_{\lt(B(0,2))}<\tilde\eps_\nu,
\end{equation}
where $\tilde\eps_\nu\to 0$ as $\nu\to\infty$.
Therefore by Lemma~\ref{lemma:fromphitopsi},
there exists a sequence $\eps_\nu^\dagger$ tending to $0$
such that
\begin{equation}
\big|
e^{i[\varphi_\nu(x)+\varphi_\nu(x')]}-c_\nu e^{i(x+x')\cdot\xi_\nu}
\big|
<\eps^\dagger_\nu,
\end{equation}
for all $(x,x')\in S^{2+2}$ except for an exceptional set, depending on $\nu$,
whose $\sigma\times\sigma$ measure tends to zero as $\nu\to\infty$.
By freezing a typical value of $x'$ and multiplying through
by $e^{-i\varphi_\nu(x')}$ we obtain
\begin{equation}
\big|
e^{i\varphi_\nu(x)}-\tilde c_\nu e^{ix\cdot\xi_\nu}
\big|
<\eps^\dagger_\nu,
\end{equation}
for all $x$ lying outside of an exceptional set whose
$\sigma$--measure tends to zero.
Here $\tilde c_\nu = c_\nu e^{ix'\cdot \xi_\nu-i\varphi_\nu(x')}$.
\qed

\begin{proof}[Proof of Lemma~\ref{lemma:fromphitopsi}]
Let $\{\rho_\nu\}$
be a sequence of positive numbers which tends to zero as $\nu\to\infty$.
Define
\begin{equation}
\scripte_{z}=\big\{(x,x')\in S^{2+2}: x+x'=z
\text{ and }
\big|
e^{i[\varphi_\nu(x)+\varphi_\nu(x')-\psi_\nu(z)]}-1
\big|
>\rho_\nu \big\}
\end{equation}
and
\begin{equation}
\scripte^\nu = \cup_{z\in B(0,2)}\scripte_z \subset S^2\times S^2.
\end{equation}
$\scripte_z$ depends on $\nu$, but this dependence is suppressed to simplify notation.

The assertion of the lemma is that if $\rho_\nu\to 0$ sufficiently slowly,
then $(\sigma\times\sigma)(\scripte^\nu)\to 0$. We will prove this by contradiction.
Thus we may assume that there exists $\rho>0$ such that if $\scripte_z,\scripte^\nu$
are redefined to be
\begin{gather}
\label{rhodefined}
\scripte_{z}=\big\{(x,x')\in S^{2+2}: x+x'=z
\text{ and }
\big|
e^{i[\varphi_\nu(x)+\varphi_\nu(x')-\psi_\nu(z)]}-1
\big|
>\rho \big\}
\\
\intertext{and}
\scripte^\nu = \cup_{z\in B(0,2)}\scripte_z \subset S^2\times S^2,
\end{gather}
then $(\sigma\times\sigma)(\scripte^\nu)\ge\rho$ for all $\nu$.

This implies that
\begin{equation}
\int_{\scripte^\nu} F_\nu(x)F_\nu(x')\,d\sigma(x)\,d\sigma(x')
\ge\rho' \text{ for all sufficiently large $\nu$}
\end{equation}
for some constant $\rho'>0$.
Indeed, by passing to a subsequence we may assume that $F_\nu\to F$ for some
nonnegative extremizer $F\in\lt(S^2)$.
By Lemma~\ref{lemma:strictlypositive}, $F>0$ almost everywhere on $S^2$.
Therefore uniformly for all sets $E\subset S^2$,
for any $\eps>0$,
$\int_E F\,d\sigma$ is bounded below by a strictly positive quantity
$\theta(\eps)$ whenever $\sigma(E)\ge\eps$.
Since $F_\nu\to F$ in $\lt(\sigma)$ norm,
it follows from Chebyshev's inequality that
for any $\eps>0$ there exists $N<\infty$ such that for every
$\nu\ge N$ and every subset $E\subset S^2$
satisfying $\sigma(E)\ge\eps$,
$\int_E F_\nu\,d\sigma\ge \tfrac12\theta(\eps)$.

In the same way it follows that for any $\eps>0$
there exist $\theta(\eps)>0$ and $N<\infty$
such that whenever $\nu\ge N$ and $E\subset S^2\times S^2$
satisfies $(\sigma\times\sigma)(E)\ge\eps$,
\begin{equation}
\int_E F_\nu(x)F_\nu(x')\,d\sigma(x)\,d\sigma(x')\ge\theta(\eps).
\end{equation}
Therefore there exists $\eta>0$ such that
\begin{equation}
\int_{\scripte^\nu} F_\nu(x)F_\nu(x')
\,d\sigma(x)\,d\sigma(x')
\ge\eta
\end{equation}
for all sufficiently large $\nu$; by discarding finitely many indices
we may assume that this holds for all $\nu$.

Recall the general formula
\begin{equation}
(h\sigma*h\sigma)(z) = c_0|z|^{-1}\int_{x+x'=z}h(x)h(x')
\,d\lambda_z(x,x'),
\end{equation}
where $c_0$ is a positive constant whose precise value is of no importance here,
and $\lambda_z$ is arc length measure on a certain (not necessarily great)
circle in $S^2\times S^2$, normalized to be a probability measure.
The push-forward from $S^2\times S^2$ to $\reals^3$
of the measure $F_\nu(x)F_\nu(x')\chi_{\scripte^\nu}(x,x') \,d\sigma(x)\,d\sigma(x')$
under the map $(x,x')\mapsto x+x'$ is equal to
\begin{equation}
G_\nu^\flat(z)=
 c_0|z|^{-1}\int_{\scripte_z}F_\nu(x)F_\nu(x')
\,d\lambda_z(x,x').
\end{equation}
Its $L^1$ norm equals the total variation measure  of
$F_\nu(x)F_\nu(x')\chi_{\scripte^\nu}(x,x')$.
Therefore
\begin{equation}
\norm{G_\nu^\flat}_{L^1(\reals^3)}
=
\int_{\scripte^\nu} F_\nu(x)F_\nu(x')
\,d\sigma(x)\,d\sigma(x')
\ge\eta.
\end{equation}

On the other hand,
since $G_\nu^\flat\le G_\nu=F_\nu\sigma*F_\nu\sigma$ pointwise,
$\norm{G_\nu^\flat}_{\lt(\reals^3)}$
is bounded above, uniformly in $\nu$.
It follows from Chebyshev's inequality that
there exists $\delta>0$
such that for every $\nu$, $G_\nu(z)\ge\delta$
for every point $z$ belonging to
a set $S_\nu\subset B(0,2)$, which satisfies $|S_\nu|\ge\delta$.

For any $z\in\reals^3$ satisfying $0<|z|<2$,
\[
(f_\nu\sigma*f_\nu\sigma)(z) = c_0|z|^{-1}\int_{x+x'=z} e^{i\varphi_\nu(x)+i\varphi_\nu(x')}
F_\nu(x)F_\nu(x')\,d\lambda_z(x,x').
\]
$e^{-i\psi_\nu(z)}
(f_\nu\sigma*f_\nu\sigma)(z)$ is real and positive by definition of $\psi_\nu$,
so
\begin{align}
\big|
(f_\nu\sigma*f_\nu\sigma)(z)\big|
&=
e^{-i\psi_\nu(z)}
(f_\nu\sigma*f_\nu\sigma)(z)
\\
&= c|z|^{-1}\int_{x+x'=z}
\Re\Big(e^{i[\varphi_\nu(x)+\varphi_\nu(x')-\psi_\nu(z)]}\Big)
F_\nu(x)F_\nu(x')\,d\lambda_z(x,x').
\end{align}
Now
\begin{equation}
\int_{\scripte_z}
\Re\Big(e^{i[\varphi_\nu(x)+\varphi_\nu(x')-\psi_\nu(z)]}\Big)
F_\nu(x)F_\nu(x')\,d\lambda_z(x,x')
\le
(1-c\rho_\nu^2)
\int_{\scripte_z}
F_\nu(x)F_\nu(x')\,d\lambda_z(x,x')
\end{equation}
for a certain positive constant $c$,
using the defining property \eqref{rhodefined} of $\rho$.
Therefore
\begin{equation}
|(f_\nu\sigma*f_\nu\sigma)(z)|
\le G_\nu(z)-c\rho^2 G_\nu^\flat(z)
\end{equation}
for all $z\in B(0,2)$,
and in particular,
\begin{equation}
|(f_\nu\sigma*f_\nu\sigma)(z)|
\le G_\nu(z)-c\rho^2 \delta
\end{equation}
for all $z\in S_\nu\subset B(0,2)$, with $|S_\nu|\ge\delta$.

Another elementary argument relying on Chebyshev's inequality
and the uniform upper bound for $\norm{G_\nu}_{\lt}$,
together with the fact that $0\le G_\nu(z)-c\rho^2 \delta\chi_{S_\nu}$,
demonstrates that
\begin{equation}
\norm{ G_\nu-c\rho^2 \delta\chi_{S_\nu}}_{\lt}
\le
\norm{ G_\nu}_{\lt}-\gamma
\end{equation}
for some positive quantity $\gamma$ which is independent of $\nu$.
Therefore
\begin{equation}
\norm{f_\nu\sigma*f_\nu\sigma}_{\lt}\le
\norm{ G_\nu}_{\lt}-\gamma
\le \sup_{\norm{f}_{\lt}\le 1} \norm{f\sigma*f\sigma}_{\lt}-\gamma
\end{equation}
for all $\nu$.
This contradicts the assumption that $\{f_\nu\}$
is an extremizing sequence, concluding the proof of the lemma.
\end{proof}

\end{document}